\documentclass{amsart}

\usepackage{mathabx}
\usepackage{extpfeil}
\usepackage{euscript}
\usepackage{hyperref}
\usepackage{amssymb, fancyhdr, stmaryrd, amsthm, mathrsfs, textcomp}

\usepackage[utf8]{inputenc}
\usepackage{amsfonts,amsmath,amscd,amssymb}
\usepackage{dsfont,mathtools}
\usepackage{xspace}
\usepackage{undertilde}
\usepackage{tikz-cd}
\usepackage{caption}
\usepackage{subcaption}
\usepackage{graphicx}
\usepackage{listings}
\usepackage{faktor}
\usepackage{calligra}

\def\N{\ensuremath\mathbb{N}}

\def\F{\ensuremath\mathbb{F}}

\def\s{\ensuremath\sigma}

\def\I{\ensuremath\mathcal{I}}
\def\J{\ensuremath\mathcal{J}}

\def\M{\ensuremath\mathcal{M}}
\def\t{\ensuremath\tau}

\def\x{\ensuremath\times}

\def\to{\ensuremath\rightarrow}
\def\sset{\ensuremath\subseteq}

\def\<{\ensuremath\langle}
\def\>{\ensuremath\rangle}

\def\quotient#1#2{%
    \raise1ex\hbox{$#1$} \big /\lower1ex\hbox{$#2$}%
}
\DeclareMathAlphabet{\mathcalligra}{T1}{calligra}{m}{n}

\DeclareMathOperator{\Hom}{Hom}
\DeclareMathOperator{\res}{Res}
\DeclareMathOperator{\aut}{Aut}

\DeclareMathOperator{\fun}{Fun}
\DeclareMathOperator{\ind}{Ind}
\DeclareMathOperator{\coind}{Coind}

\newtheorem{thm}{Theorem}[section]
\newtheorem*{thm*}{Theorem}
\newtheorem{lemma}[thm]{Lemma}

\theoremstyle{definition}
\newtheorem{defn}[thm]{Definition}

\newtheorem{ex}[thm]{Example}

\newcommand{\vg}{\ensuremath{\mathbf{g}}\xspace}
\newcommand{\vh}{\ensuremath{\mathbf{h}}\xspace}

\newcommand{\vk}{\ensuremath{\mathbf{k}}\xspace}

\newcommand{\vx}{\ensuremath{\mathbf{x}}\xspace}
\newcommand{\vy}{\ensuremath{\mathbf{y}}\xspace}

\newcommand{\sB}{\mathcal{B}}
\newcommand{\sC}{\mathcal{C}}
\newcommand{\sD}{\mathcal{D}}

\newcommand{\sF}{\mathcal{F}}

\newcommand{\sJ}{\mathcal{J}}

\newcommand{\sK}{\mathcal{K}}

\newcommand{\sT}{\mathcal{T}}

\newcommand{\sU}{\mathcal{U}}

\begin{document}
\title[Frobenius reciprocity]{Frobenius reciprocity for topological groups}
\author{Katerina Hristova}
\email{K.Hristova@warwick.ac.uk}
\address{Department of Mathematics, University of Warwick, Coventry, CV4 7AL, UK}
\thanks{The research is supported by the author's EPSRC PhD studentship. The author would like to thank Dmitriy Rumynin for proposing the problem and for the numerous valuable suggestions and helpful comments, and also Inna Capdeboscq for the many helpful conversations.}
\date{7  January 2018}
\subjclass{Primary 22A25 ; Secondary 22D30}
\keywords{induction, coinduction, restriction, left adjoint, right adjoint, continuous representation, Tate module, compact representation, linearly compact topological module}

\begin{abstract}
We investigate the existence of left and right adjoints to the restriction functor in three categories of continuous representations of a topological group: discrete, linear complete and compact.

\end{abstract}

\maketitle

\section*{Introduction}

Given a representation of a group $G$ one can always define a representation of a subgroup $H \leq G$ by restricting the group action to the subgroup. This defines a functor 
$$\res_H^G: \mathcal{R}ep(G) \to \mathcal{R}ep(H).$$
We are interested in functors going in the opposite direction which are adjoint to $\res_H^G$. Such functors are called induction functors and the adjointness relation is known as Frobenius reciprocity. As the restriction functor is a forgetful functor, it has both a left and a right adjoint given by $\otimes$ and $\Hom(,)$ respectively. Thus, we have two induction functors:
$$\ind_H^G, \coind_H^G: \mathcal{R}ep(H) \to \mathcal{R}ep(G).$$
To distinguish between them we call the right adjoint induction and denote it $\ind_H^G$, and the left adjoint coinduction and denote it $\coind_H^G$. 
 Since adjoints are unique up to isomorphism, it is clear that in the case of groups $\coind_H^G$ is given by the tensor product and $\ind_H^G$ is given by the $\Hom(,)$ functor. For finite groups, and more generally for groups of arbitrary cardinality, provided that $H$ is of finite index in $G$, $\ind_H^G$ and $\coind_H^G$ are isomorphic. However, in general this is not true. It is worth noting that in some of the literature the opposite convention is adopted - the left adjoint functor is called induction and the right - coinduction. This is always the case when studying the representation theory of rings. However, for $p$-adic groups the notions are the same as ours which is the reason for our terminology.

In the present paper we wish to extend the idea of Frobenius reciprocity to topological groups and to investigate the existence of the induction and coinduction functors. In particular, we fix a topological group $G$ and a closed subgroup $H \leq G$. We look at \emph{continuous} representations of such groups over an associative ring $R$. These are topological modules $(V, \sT)$ over $R$, such that the action map $G \x V \to V$ is continuous with respect to the topology $\sT$ on $V$. Varying this topology we obtain different categories of continuous representations. We are interested in three such: the category of discrete representations $\M_d(G)$, where $V$ is endowed with the discrete topology, the category of linearly topologized and complete representations $\M_{ltc}(G)$, where the topology on $V$ is linear and complete, and the category of compact representations $\M_c(G)$, where $V$ is given a linear, complete topology in which all quotients of open submodules are of finite length. In each of these categories we investigate the existence of left and right adjoint to the restriction functor $\res_H^G$. Our main tool is Freyd's Adjoint Functor Theorem. It gives a necessary and sufficient condition for existence of adjoints in (locally) small categories \cite{mac}. Each section of the paper contains a main result which is a criterion for the existence of a left and a right adjoint to $\res_H^G$, i.e, we give variants of Frobenius reciprocity in each of the aforementioned categories.
\begin{thm*}[Main]
Let $G$ be a topological group and $H \leq G$ a closed subgroup. The restriction functor 
$$\res_H^G: \M_{\star}(G) \to \M_{\star}(H)$$
has the following properties:
\begin{enumerate}
\item In $\M_d(G)$ the functor $\res_H^G$ always has a right adjoint, given by the induction functor $\ind_H^G$, and has a left adjoint $\coind_H^G$ if $H$ is also open.
\item In $\M_{ltc}(G)$ and $\M_c(G)$ the functor $\res_H^G$ always has a left adjoint $\coind_H^G$ and has a right adjoint $\ind_H^G$ if $H$ is also open.
\end{enumerate}

\end{thm*}
Now let us give a section-by-section outline of the present paper.

Section \ref{one} is a short introduction to the categories of interest. We set the notation and give a brief description of the objects and morphisms in each category. A more precise definition of a continuous representation of $G$ is given. 

In Section \ref{smooth} we investigate Frobenius reciprocity in the category of discrete representations $\M_d(G)$ of $G$. Considering representations over modules endowed with the discrete topology is the standard approach to continuous representations. In particular, if $G$ is a locally compact totally disconnected group, then $\M_d(G)$ is precisely the category of smooth representations of $G$ \cite{LL}, \cite{Vig1}, \cite{Vig2}. This category is widely studied as examples of groups with such topology include $p$-adic groups, topological Kac-Moody groups and groups of Kac-Moody type \cite{HrD}. 
The construction of the induction functor from a closed subgroup $H \leq G$ in the case of smooth representations is well-known. Following Bushnell and Henniart in terminology, we generalise the construction of $\ind_H^G$ to arbitrary topological groups \cite{LL}.
Next, we move on to coinduction, which is more subtle. In the case of locally compact totally disconnected groups, the coinduction functor is called compact induction or induction with compact support \cite{LL}, \cite{Vig1}. In Theorem \ref{l-adj d} we establish a sufficient condition for its existence in $\M_d(G)$. In particular, we claim that $H\leq G$ must also be an open subgroup. 
 It is worth remarking that Bushnell and Henniart work over 
a field of characteristic zero. This is extended by Vign\'{e}ras to positive characteristic. She also describes the induction and compact induction functors from a closed subgroup of a locally profinite group for modules over a commutative ring \cite{Vig1}. We, however, keep the generality of an associative ring with identity and arbitrary topological groups.

In Section \ref{two} we move on to the category of linearly topologized and complete continuous $G$-modules $\M_{ltc}(G)$. First, we give a precise formulation of the continuity condition. We then begin our investigation of adjoints to $\res_H^G$. We wish to use Freyd's Theorem to establish whether $\ind_H^G$ and $\coind_H^G$ exist in $\M_{ltc}(G)$. The main condition, which the category needs to satisfy to ensure such existence, is to be complete and cocomplete, i.e., to be closed under taking small limits and small colimits. Since $\M_{ltc}(G)$ is abelian, using the product-equalizer construction of limits and the coproduct-coequalizer construction of colimits, it is enough to show that the category is closed under taking products and coproducts. We show how to construct these in $\M_{ltc}(G)$ in Lemma \ref{prod} and Lemma \ref{coprod} respectively. The main result of Section \ref{two} is Theorem \ref{adjoints-ltc}. It establishes the existence of the left adjoint to $\res_H^G$, i.e., $\coind_H^G$ in the case when $H \leq G$ is closed, and the existence of the right adjoint, i.e., $\ind_H^G$ when $H$ is also open.

In Section \ref{three} we describe the category of (linearly) compact continuous representations of $G$, denoted $\M_c(G)$. The notion of linear compactness of vector spaces first appears in Lefschetz' ``Algebraic Topology'' \cite{Lef}. He calls a vector space linearly compact if it is linearly topologized, Hausdorff, and satisfies the finite intersection property on cosets of closed subspaces. Such spaces are complete \cite{Lef}. This leads to an alternative definition of linearly compact vector spaces given by Drinfeld as linearly topologized complete Hausdorff spaces with the property that open subspaces have finite codimension \cite{Dri}. Dieudonn\'{e} unifies these definitions by showing they are all equivalent in the case of vector spaces \cite{Dieu}. Compact vector spaces are also topological duals of discrete ones \cite{Dri}. Taking the duality viewpoint, we begin the section by constructing an example to show that if $R$ is a field, then $\res_H^G$ is cocontinuous in neither $\M_c(G)$, nor $\M_{ltc}(G)$ unless $H$ is open.
We then move on to the case of modules over an associative ring. In this setting the definitions for compactness given above are not equivalent. A linearly compact topological $R$-module $V$ is linearly topologized, Hausdorff, and such that every family of closed cosets in $V$ has the finite intersection property \cite{War}. However, this is not  equivalent to $V$ being linearly topologized, complete, and such that open submodules have finite colength. We wish to take the point of view of the latter definition as it is closer to the Beilinson-Drinfeld approach to linearly compact topological vector spaces \cite{Hitch}.
Modules defined as above are known in the literature as \emph{pseudocompact} \cite{Iovanov}, \cite{frauke}. These come up in deformation theory, in particular, they are useful when describing lifts and deformations of representations of a profinite group over a perfect field of characteristic $p$ \cite{frauke}. 

 After we have fixed the definition of the compact topology for a module over a ring, we follow our strategy from Section \ref{two}: we construct products (Lemma \ref{prod-c}), coproducts (Lemma \ref{coprod-c}) and investigate the existence of $\ind_H^G$ and $\coind_H^G$ in $\M_c(G)$. Theorem \ref{adj-c} is the main result of the section. It establishes the existence of the coinduction functor for $H \leq G$ closed, and the existence of the induction functor, given that $H$ is also open. 


We finish the paper with a brief discussion of the category of Tate representations $\M_T(G)$. Tate spaces, or locally linearly compact spaces as defined by Lefschetz \cite{Lef}, are complete linearly topologized vector spaces, such that the basis at zero is given by mutually commensurable subspaces \cite{Beil}. Equivalently, a Tate space is a vector space which splits as a topological direct sum of a discrete and a compact space \cite{Dri}. The latter definition also generalises to modules over a commutative ring \cite{Hitch}. Hence, for a topological group $G$ one can define the category of \emph{Tate representations}, as the category with objects Tate spaces, on which $G$ acts continuously. These are an interesting object to study as they appear not only in the phenomenal work of Tate, but also in other areas of mathematics, such as the algebraic geometry of curves, the study of chiral algebras and infinite dimensional Lie algebras, as well as in Conformal Theory \cite {Hitch}, \cite{Beil}.  We do not fully investigate the analogue of Frobenius reciprocity in $\M_T(G)$, but we pose some questions about it.

\section{Introducing the categories} \label{one}

Throughout let $R$ be an associative ring with $1$ and $G$ a topological group. We are interested in studying \emph{continuous} representations of $G$ over $R$.  Let us explain precisely what we mean by this. 

First, recall that $(V, \sT)$ is called a \emph{topological $R$-module} if $\sT$ makes $(V,+)$ into a topological group and the $R$-action map $\cdot: R \x V \to V, (r, v) \mapsto r \cdot v$ is continuous with respect to $\sT$ on the right and the product topology on the left (where $R$ is endowed with the discrete topology). 
With this in mind, we make the following definition:

\begin{defn}
Let $(V, \sT)$ be a topological (left) $R$-module and $\pi: G \to \aut_R(V)$ a homomorphism. Then the pair $(\pi, V)$ is a representation of $G$. It is called \emph{continuous} if the map $\phi: G \x V \to V$ defined by $(\vg,v) \mapsto \pi(\vg)v$ is continuous with respect to the product topology on the left and $\sT$ on the right.
\end{defn}
Whenever we talk about topological $R$-modules, we always mean left modules, but of course the results remain true for right $R$-modules. From the definition above it is clear that the continuity condition depends on the topology we put on $V$. Hence, by changing this topology we obtain different categories of continuous representations. We are mainly interested in three such:
\begin{itemize}
\item $\M_d(G)$ - category of discrete representations of $G$. The objects are continuous representations $(\pi, V)$ of $G$, such that $(V, \sT)$ is a topological $R$-module, endowed with the discrete topology. The morphisms between two objects $(\pi_1, V_1)$ and $(\pi_2, V_2)$ are given by $R$-linear maps $f: V_1 \to V_2$, such that $f(\pi_1(\vg)v)=\pi_2(\vg) f(v)$, for $\vg \in G$ and $v \in V_1$. 
\end{itemize}
In the next two categories of interest $V$ is  given a \emph{linear} topology $\sT$. More precisely, we say that a topology $\sT$ is \emph{linear}, or that $V$ is \emph{linearly topologized}, if the open $R$-submodules of $V$ form a fundamental system of neighbourhoods at zero \cite{War}. 
This gives rise to the following categories:
\begin{itemize}
\item $\M_{ltc}(G)$ - category of linearly topologized complete representations. The objects are pairs $(\pi, V)$, where $V$ is a continuous representation of $G$, endowed with a linear topology $\sT$, such that $(V, \sT)$ is a complete topological space. The morphisms between two objects $(\pi_1, V_1)$ and $(\pi_2, V_2)$ are given by continuous $R$-module homomorphisms $f: V_1 \to V_2$, such that $f(\pi_1(\vg)v)=\pi_2(\vg) f(v)$, for $\vg \in G$ and $v \in V_1$. 

\item $\M_c(G)$ - category of compact representations. The objects are pairs $(\pi, V)$ where $(V, \sT)$ is a linearly topologized complete $R$-module, such that for every open $R$-submodule $W \leq V$, $V/W$ is an $R$-module of finite length. The morphisms are defined in the same way as in $\M_{ltc}(G)$.
\end{itemize}
We give further details on the topologies of the three categories defined above, as well as the explicit meaning of the continuity condition, in the sections to follow.

\section{Category of Discrete Representations}\label{smooth}

Fix a topological group $G$ and a closed subgroup $H \leq G$. We study the category $\M_d(G)$ of discrete representations of $G$ and $\M_d(H)$ of discrete representations of $H$. These are connected by the restriction functor:
$$ \res_H^G: \M_d(G) \to \M_d(H),\ \text{defined by}\ (\pi, V) \mapsto (\pi\mid_H, {}_{H}V), $$
where $\pi\mid_H: H \to \aut_R(V)$ is the restriction of $\pi: G \to \aut_R(V)$ to $H$ and ${}_{H}V$ denotes $V$ as an $H$-module.

Let $(\pi, V)$ be a representation of $G$ and $\phi: G \x V \to V$, given by $(\vg,v) \mapsto \pi(\vg)v$, be the map induced by the action of $G$. 
A discrete representation $(\pi, V)$ of $G$ is \emph{continuous} if for every $v \in V$, $\phi^{-1}(v)$ is open in $G$. In other words, for every $v \in V$, there exists an open set $K_v \subset G$, such that $\pi(\vk)v=v$, for every $\vk \in K_v$.

Note that $1_G$ always satisfies $\pi(1_G)v=v$.  Since the group topology is determined by the fundamental neighbourhoods of identity, without loss of generality assume that $K_v$ is an open neighbourhood of $1_G$. We could go even further - for every $v \in V$ we can construct an open subgroup $\widetilde{K_v} \leq G$ generated by $K_v$. Then clearly $\pi(\vk)v=v$, for every $\vk \in \widetilde{K_v}$.


 
If the topology of $G$ is locally compact and totally disconnected, then $\M_d(G)$ is the category of smooth representations of $G$. The smoothness condition there states that Stab$_G(v)$ is open in $G$ for every $v \in V$, which is precisely our continuity condition. 
 
The main goal of this section is to determine when the restriction functor $\res_H^G$ has a left and a right adjoint in $\M_d(G)$. 

We start by looking at the right adjoint to $\res_H^G$.
We claim that it exists and is given by the induction functor $\ind_H^G: \M_d(H) \to \M_d(G)$. We define $\ind_H^G$ generalising the construction of smooth induction for locally compact totally disconnected groups \cite{LL}:

Fix $(\s, W) \in \M_d(H)$. Consider the $R$-module $\widehat{W}$ of all left $H$-equivariant functions $f: G \to W$, i.e., which satisfy the property 
\begin{itemize}
\item[(i)] $f(\vh \vg)=\s(\vh)f(\vg)$, for all $\vh \in H$ 
and $\vg \in G$.
\end{itemize}
Within $\widehat{W}$ we find an $R$-submodule $\widetilde{W}$ consisting of ``continuous functions'', i.e., functions with the additional property
\begin{itemize}
\item[(ii)] $f \in \widetilde{W}$
if and only if there exists an open neighbourhood $K_f$ of $1_G$, such that $f(\vg \vk)=f(\vg)$, for all $\vg \in G$ and $\vk \in K_f$.
\end{itemize}
The homomorphism 
$\rho: G \to\mbox{Aut}_R(\widehat{W})$, given by $\rho(\vg)f: \vx \mapsto f(\vx \vg)$, for $\vg,\vx \in G$ and $f \in \widehat{W}$, defines a $G$-action on both $\widehat{W}$ and $\widetilde{W}$, thus making $(\rho, \widetilde{W})$ a continuous representation of $G$, i.e., $(\rho, \widetilde{W}) \in \M_d(G)$. The pair 
$(\rho, \widetilde{W})$ is called the representation of $G$ 
\emph{continuously induced by $\s$} and is denoted $\ind_{H} ^{G} (\s)$. Using this construction we define the functor $\ind_H^G: \M_d(H) \to \M_d(G)$. We claim that this is the right adjoint we are looking for. 

\begin{lemma}
For a topological group $G$ and a closed subgroup $H \leq G$, the functor $\ind_H^G: \M_d(H) \to \M_d(G)$ defined above is right adjoint to the restriction functor.
\end{lemma}

\begin{proof}
For continuous representations $(\pi,V) \in \M_d(G)$ and $(\s, W) \in \M_d(H)$ and notation as above, we want
$$\Hom_G(V, \widetilde{W}) \cong \Hom_H({}_{H}V, W).$$
We have maps:
$\alpha: \Hom_G(V, \widetilde{W}) \to \Hom_H({}_{H}V, W)$ given by:
$$ \psi: V \to \widetilde{W}  \mapsto \quad \widetilde{\psi}:V \to W $$
  $$  \psi:  v \mapsto \psi_v \quad \mapsto \quad \widetilde{\psi}:  v \mapsto \psi_v(1_G),$$
 
and $\beta: \Hom_H(V, W) \to \Hom_G(V, \widetilde{W})$ given by:
$$ \phi: V \to W  \mapsto \quad \widetilde{\phi}:V \to \widetilde{W} $$
 with
$$ \widetilde{\phi}:  v \mapsto f_v: \vg \mapsto \phi(\vg \cdot v).$$
Clearly $\widetilde{\psi} \in \Hom_H (V, W)$ since $\psi_v \in \widetilde{W}$.
Similarly for $\widetilde{\phi} \in \Hom_G(V, \widetilde{W})$. It is routine to check that $\alpha$ and $\beta$ are inverse to each other.
\end{proof}




Now we move on to the case of the left adjoint. This is more subtle. Let us lay out our conventions first. We use the following standard terminology: 
\begin{itemize}
\item A functor $\sF: \sC \to \sD$ is called: 
\begin{itemize}
\item \emph{continuous} if it preserves small limits,
\item \emph{cocontinuous} if it preserves small colimits.
\end{itemize}
\item A category $\sC$ is called:
\begin{itemize}
\item \emph{complete} if all small diagrams have limits in $\sC$, 
\item \emph{cocomplete} if all small diagrams have colimits in $\sC$.
\end{itemize}
\end{itemize}


Since limits can be constructed as equalizers of products, a category $\sC$ is complete if all morphisms in $\sC$ have equalizers and $\sC$ is closed under arbitrary products \cite{mac}. Hence, to check continuity of a functor $\sF: \sC \to \sD$, it is sufficient to check that $\sF$ preserves those (respectively coproduct and coequalizers for the cocontinuous case).

Recall the following criterion for existence of a left adjoint to a functor \cite{mac}:

\begin{thm*}[The Freyd Adjoint Functor Theorem]
Given a locally small, complete category $\sC$ 
a functor $\sF: \sC \to \sD$ has a left adjoint if and only if it preserves all small limits and satisfies the following condition:
\begin{itemize}
\item[(SSC)] For each object $d \in \sD$ there is a small set $\I$ and an $\I$-indexed family of morphisms $f_i: d \to \sF(c_i)$, such that every morphism $h: d \to \sF(c)$ can be written as a composite $h=\sF(t) \circ f_i$, for some index $i \in \I$ and some $t: c_i \to c$.
\end{itemize}
Dualise the statement to obtain a criterion for a right adjoint.
\end{thm*}

We wish to use Freyd's Theorem to determine whether the restriction functor has a left adjoint. First note that (SSC) holds in $\M_d(G)$: it just says that every map in $\M_d(G)$ can be factored through a quotient.
Also note that $\M_d(G)$ is abelian, so equalizers of all morphisms exist. Since $\res_H^G$ does not change the morphisms between the objects, it commutes with equalizers. 

The next step is to check whether $\M_d(G)$ is closed under arbitrary products. Take a collection $\{V_i\}_{i \in \I} \in \M_d(G)$, for some arbitrary set $\I$.
Let $V \coloneq \prod\limits_{i \in \I} V_i$ denote the product of $V_i$ as $R$-modules. $V$ remains a discrete space with respect to the box topology. It also has an obvious $G$-module structure - $G$ acts componentwise:
$$ \vg \cdot (v_1,v_2,...,v_n,..)=(\vg \cdot v_1,\vg \cdot v_2,...,\vg \cdot v_n,..),\ \text{for}\ \vg \in G, v_i \in V_i.$$
However, the action is not necessarily continuous:

Fix $v \coloneq (v_1,v_2,...,v_n,..) \in V$. Since $V_i \in \M_d(G)$, for every $v_i \in V_i$, there exists an open neighbourhood $K_{v_i}$ of $1_G$, such that $\vk \cdot v_i=v_i$, for all $\vk \in K_{v_i}$ and $i \in \I$. Thus, $K_v=\bigcap_{i \in \I} K_{v_i}$ has the property that $\vk \cdot v=v$, for all $\vk \in K_v$. But as $\I$ is chosen arbitrarily $K_v$ does not have to be open. Therefore, the representation is not continuous at $v$ and $V \notin \M_d(V)$. However, consider the \emph{continuous} part of $V$, i.e., 
$$V^{sm} \coloneq \{v \in V\ |\ \text{there exists}\ K_v\ \text{open in}\ G,\ \text{such that}\ \vk \cdot v=v,\ \text{for all}\ \vk \in K_v\}.$$
Clearly $V^{sm}$ is a continuous representation of $G$. We claim the following:

\begin{lemma}
Every collection $\{V_i \}_{i \in \I} \in \M_d(G)$, $\I$ an arbitrary set, has a product in $\M_d(G)$ given by $(\prod\limits_{i \in \I} V_i)^{sm}$. In other words, $\M_d(G)$ is complete.
\end{lemma}

\begin{proof}
Let $V^{sm} \coloneq (\prod\limits_{i \in \I} V_i)^{sm}$. Denote by $p_i: V^{sm} \to V_i$ the canonical projections in $\M_d(G)$. Let $A \in \M_d(G)$ and $f_i: A \to V_i$ be a family of morphisms in $\M_d(G)$ indexed by $\I$. As $V^{sm} \leq \prod_{i \in \I} V_i \in R-$\textbf{mod}, there exists a unique $R$-module homomorphism $f: A \to V^{sm}$, such that $p_i \circ f=f_i$, for all $i \in \I$. It is also a $G$-map:
$$p_i(f (\vg \cdot a))=f_i(\vg \cdot a)=\vg \cdot f_i(a)=\vg \cdot (p_i(f(a))=p_i(  \vg \cdot (f(a)),\ \vg \in G, a \in A,$$
i.e.,
$$f(\vg \cdot a)= \vg \cdot f(a).$$ Thus, $f$ is a morphism in $\M_d(G)$ and the universal property of the product is satisfied.
\end{proof}
We claim that even though $\M_d(G)$ is complete, the restriction functor does not always have a left adjoint.

\begin{thm}\label{l-adj d}
Let $G$ be a topological group and $H$ a closed subgroup of $G$. Then 
$$ \res_H^G: \M_d(G) \to \M_d(H) $$
has a left adjoint if $H$ is also open. 
\end{thm}

\begin{proof}
Taking into consideration the discussion before Lemma \ref{prod}, to establish the existence of a left adjoint, we have to show that $\res_H^G$ is continuous, i.e.:
$$\bigg( \prod\limits_{i \in \I} \res_H^G(V_i) \bigg)^{sm} \cong \res_H^G \big( (\prod\limits_{i \in \I} V_i)^{sm} \big),\ \text{for some indexing set}\ \I.$$
This is the same as showing that every $W \in  \M_d(H)$, such that $W=\big(\prod \res_H^G(V_i) \big)^{sm}$, for some $V_i \in \M_d(G)$, is also an element of $\M_d(G)$.

Take such $W \in \M_d(H)$. Then for every $w \in W$ there exists an open neighbourhood $K_w$ of $1_H$, such that $\vk \cdot w=w$, for all $\vk \in K_w$. As $H$ is given the subspace topology, there is an open $U \sset G$, such that $K_w=U \cap H$. But since $H$ is open in $G$, then $K_w$ is an open neighbourhood of $1_G$ and thus $W \in \M_d(G)$. 
\end{proof}
Thus, for an open subgroup $H \leq G$ , we have a functor
$$\coind_H^G: \M_d(H) \to \M_d(G),$$
which is left adjoint to the restriction functor.

\begin{ex}
Suppose the topology on $G$ is locally compact and totally disconnected. Then $\coind_H^G$ is given by compact induction of representations, i.e.,
$$\coind_H^G: \M_d(H) \to \M_d(G), \quad (\s,W) \mapsto (\rho, \widetilde{W}),$$
where $\widetilde{W}$ is the space of all left $H$-equivariant, continuous functions $f: G \to W$, which have compact modulo $H$ support, and $\rho( \vg) f: \vx \to f(\vx \vg)$ \cite{LL}. It is clear from the constructions that for a locally compact totally disconnected $G$ and $H$, such that $H \setminus G$ is compact, $\ind_H^G \cong \coind_H^G$ \cite{LL}, \cite{Vig1}.  
\end{ex}

The next example shows that if $H$ is not open in $G$, then $\coind_H^G$ is not always defined.
\begin{ex}
Take $\{ V_i \}_{i \in \I} \in \M_d(G)$. Fix $v=(v_1,v_2,...,v_n,..) \in \prod\limits_{i \in \I}V_i $ with $K_i$ corresponding open neighbourhoods of $1_G$, such that $\vk_i \cdot v_i=v_i$, for all $\vk_i \in K_i$. For each $i \in \I$ construct an open subgroup $\widetilde{K_i} \leq G$, generated by $K_i$. Then $\vk \cdot v=v$, for all $\vk \in \widetilde{K} \coloneq \bigcap_{i \in \I} \widetilde{K_i}$. However, as $\I$ is an arbitrary set $\widetilde{K}$ is not necessarily open in $G$. Moreover, it is closed as every $\widetilde{K_i}$ is. Taking $H \coloneq \widetilde{K}$ we have $v \notin (\prod\limits_{i \in \I}V_i)^{sm}$, and thus $v \notin \res_H^G (\prod\limits_{i \in \I}V_i)^{sm}$. But as $H$ is open in $H$, $v \in ( \prod\limits_{i \in \I}( \res_H^G(V_i) )^{sm}$. Thus, $\res_H^G$ fails to be continuous and $\coind_H^G$ is not defined.
\end{ex}




\section{Linearly topologized and complete $G$-modules} \label{two}

Let $R$ be an associative ring with $1$, $V$ a topological $R$-module and $G$ a topological group. In this section we investigate the category $\M_{ltc}(G)$ of linearly topologized and complete $R$-modules which admit a continuous action of $G$.


Let $(\pi, V) \in \M_{ltc}(G)$. Then the map $\phi: G \x V \to V$ defined by $\phi: (\vg, v) \mapsto \pi(\vg)v$ is continuous. 
In particular, for an open $U \sset V$ there exists an open neighbourhood $K$ of $1_G$ and an open submodule $W \leq V$, such that if $\vg \cdot x \in U$, then $K\vg \cdot (x +W) \sset U$, for some $\vg \in G$ and $x \in V$.


Fix a closed subgroup $H \leq G$. As in Section \ref{smooth}, given a continuous representation $(\pi, V)$ of $G$, we obtain a representation of $H$ by restricting the map $\phi: G \x V \to V$ to $H$. As a restriction of a continuous map it remains continuous. Thus, once again we have a functor
$$\res_H^G: \M_{ltc}(G) \to \M_{ltc}(H).$$
We wish to investigate the existence of adjoints to this functor. We start with the left adjoint. Following the same strategy as in Section \ref{smooth} we begin by constructing arbitrary products in $\M_{ltc}(G)$.

\begin{lemma} \label{prod}
Arbitrary products exist in $\M_{ltc}(G)$. More precisely, the product of a collection of objects of $\M_{ltc}(G)$ is their product  
as $R$-modules, endowed with the product topology. 
\end{lemma} 

\begin{proof}
For an arbitrary collection $\{(\pi_i, V_i) \}_{i \in \I}$ of elements of $\M_{ltc}(G)$, let $V \coloneq \prod_{i \in \I} V_i$ denote their product in \textbf{$R$-mod} which is a topological $R$-module with respect to the product topology \cite{Bou}.

Following Lefschetz we show that the topology on $V$ is linear \cite{Lef}. Let $\{U^j_i\}, j \in \J_i$ be a base of neighbourhoods of $0$ in $V_i$, consisting of open submodules. Then for any finite subset $\mathcal{K} \subset \I$, $\prod_{k \in \mathcal{K}} U^j_k \x \prod_{i \in \I \setminus \mathcal{K}} V_i$ is a base of  neighbourhoods of $0$ in $V$ consisting of open submodules, giving the linearity of the product topology.


Let $z_n=(z_n^1,z_n^2,..,z_n^i,...), z_n^i \in V_i$ and $n \in \mathbb{L}$ for an ordinal number $\mathbb{L}$, be a Cauchy net in $V$. Since all $V_i$ are complete, $z^i_n$ is a convergent Cauchy net in $V_i$. Let
$$z^i \coloneq \lim_{n \in \mathbb{L}} z^i_n.$$
Set $z=(z^1,z^2,...,z^i,...)$. Let $U_i$ be an open neighbourhood of $z^i$ in $V_i$ and define $U= \prod_{i \in \J} U_i \x \prod_{i \in I \setminus \J} V_i$, for some finite subset $\J \subset \I$. Since each net $z^i_n$ is convergent there exists some $l_i$, such that $z^i_n \in U_i$, for all $n \geq l_i$ in $\mathbb{L}$. Pick the largest $l_i, i \in \J$, say $l$. Then for all $n \geq l$ in $\mathbb{L}$, $ z^i_n \in U_i$ for all $i \in \I$. Thus, $z \in U$ and $z_n$ is convergent in $V$ with $$\lim_{n \in \mathbb{L}}z_n=z.$$


The $G$-action on $V$ is componentwise. We want to show it is continuous.
Let  $U \coloneq \prod_{i \in \J} U_i \x \prod_{i \in \I \setminus \J} V_i$ with $\J \subset \I$ finite. Then $U \sset V$ is open.
 

Since all $V_i$ are continuous $G$-modules, for every $i \in \I$ there exists an  open neighbourhood $N_i$ of $1_G$ and an open submodule $W_i \leq V_i$, such that if $\vg \cdot x_i \in U_i$, for $\vg \in G, x_i \in V_i$, then $N_i \vg \cdot(x_i + W_i) \sset U_i$. Fix $N \coloneq \bigcap_{i \in \J} N_i$. Since $\J$ is finite, $N$ is an open neighbourhood of $1_G$ and furthermore $N \vg \cdot (x_i + W_i) \sset U_i$, for all $i \in \J$.
Let $W \coloneq \prod_{i \in \J} W_i \x \prod_{i \in \I \setminus \J} V_i$. This is an open submodule of $V$. Thus, we found $N \sset G$ and and open submodule $W \leq V$, such that for $\vg \cdot x \in U$, with $x=(x_1,..,x_i,..)$, $N \vg \cdot (x+W) \in U$. Hence, $V \in \M_{ltc}(G)$.

Let $A \in \M_{ltc}(G)$, $p_i: V \to V_i$ be the projections in $\M_{ltc}(G)$ and $f_i: A \to V_i$ be a family of morphisms in $\M_{ltc}(G)$ indexed by $\I$. As $V$ is the product of $V_i$ in $R$-\textbf{mod}, there exists a unique $R$-module homomorphism $f: A \to V$, making the following diagram commute:
\begin{center}
\begin{tikzcd}
  V \arrow[r, "p_i"] 
    & V_i \\
&A  \arrow[ul, dashrightarrow, "f"] \arrow[u, "f_i"]
\end{tikzcd}
\end{center}
The map $f$ has the following properties:
\begin{enumerate}
\item $ (p_i \circ f)(\vg \cdot a)=f_i(\vg \cdot a)=\vg \cdot f_i(a)=\vg \cdot (p_i \circ f)(a),\ \text{for}\ \vg \in G\ \text{and}\ a \in A$, i.e., $f$ is $G$-linear.
\item For $U \sset V$ open, $U_i \coloneq p_i(U) \sset V_i$ is open. By continuity of $f_i$ it follows that $f_i^{-1}(U_i) \sset A$ is open, for every $i \in \I$. Thus, $f^{-1}(U)=f_i^{-1}(p_i(U)) \sset A$ is open, showing that $f$ is continuous.
\end{enumerate}

Thus, $f$ is a morphism in $\M_{ltc}(G)$, finishing the proof.
\end{proof}

To continue our investigation of adjoint functors, we would also need existence of arbitrary coproducts in $\M_{ltc}(G)$. We construct them explicitly. Let $\{V_i\}_{i \in \I}$ be an arbitrary collection of elements in $\M_{ltc}(G)$. Denote by $V \coloneq \bigoplus_{i \in \I} V_i$ their coproduct in \textbf{$R$-mod} and by $\alpha_i: V_i \to V$ the canonical injections. In this case they are just inclusion maps. 
We follow Higgins in defining the topology on $V$ \cite{Hig}:

Consider pairs $(W, \tau_W)$, such that:
 \begin{enumerate}
 \item $W \in \M_{ltc}(G)$, such that there exists a surjective $R$-module homomorphism $q_w: V \to W$, which is also $G$-linear,
 \item $\tau_W$ is a topology on $W$ in which the maps $q^i_W: V_i \to W$ that factor through $q_w$ are continuous.
\end{enumerate} 
All such pairs $(W, \tau_W)$ taken up to isomorphism form a set. Hence, we can form a product $\prod_{(W, \tau_W)}W$. The map 
 \begin{equation}\label{embedding}
 q: V \to \prod_{(W, \tau_W)}W,\ \text{given by}\ v \mapsto (q_w(v))_{(W, \tau)}
 \end{equation}
is an embedding. We endow $\prod_{(W, \tau_W)}W$ with the product topology and $V$ with the topology induced by $q$. This is a group topology \cite{Hig}. 
The map $\phi: R \x \prod_{(W, \tau_W)}W \to \prod_{(W, \tau_W)}W$ is continuous, hence, the restriction $\phi\mid_{q(V)}: R \x q(V) \to q(V)$ is also continuous. Thus, the subspace topology on $q(V)$, and respectively the induced one on $V$, is an $R$-module topology. By Lemma \ref{prod} $\prod_{(W, \tau_W)}W$ lies in $\M_{ltc}(G)$. Every subspace of a linearly topologized space is linearly topologized \cite{Lef}. Thus, as $V \cong q(V)$, both as an $R$-module and as a topological space, the topology on it is linear. A priori $V$ is not necessarily complete. However, its closure $\widebar{V}$ is, as it is a closed subspace of a complete space \cite{Bou}.

\begin{lemma} \label{coprod}
$\widebar{V}$ as defined above is the coproduct of $\{V_i\}_{i \in \I}$ in $\M_{ltc}(G)$.
\end{lemma}

\begin{proof}
By definition $\widebar{V}$ is a linearly topologized and complete space. As $V=\bigoplus_{i \in \I} V_i$ and $V_i$ is a $G$-module for every $i \in \I$, then clearly so is $V$. Since $\prod_{(W, \tau_W)}W \in \M_{ltc}(G)$, the map $G \x \prod_{(W, \tau_W)}W \to \prod_{(W, \tau_W)}W$ is continuous. Hence, its restriction to a subspace is also continuous.
Therefore, $V$ is a continuous $G$-module and hence, so is its closure $\widebar{V}$. Thus, $\widebar{V} \in \M_{ltc}(G)$ as required.
 
Let us check that $\bar{V}$ is indeed the coproduct of $\{V_i\}_{i \in \I}$. Let $A$ be any module in $\M_{ltc}(G)$ and $\beta_i:V_i \to A$ be morphisms in $\M_{ltc}(G)$ indexed by $\I$. Since $V$ is the coproduct of $\{ V_i \}_{i \in \I}$ in \textbf{$R$-mod}, 
there exists a unique $R$-linear homomorphism $f: \widebar{V} \to A$, such that for every $i \in \I$ the diagram below commutes:
\begin{center}
\begin{tikzcd}
  V_i \arrow[r, "\alpha_i"] \arrow[dr, "\beta_i"]
    & \widebar{V} \arrow[d, dashrightarrow, "f"]\\
&A \end{tikzcd}
\end{center}
The map $f$ is $G$-linear:
$$ f \alpha_i (\vg \cdot v_i)=f(\vg \cdot (\alpha_i(v_i)))=\beta_i( \vg \cdot v_i)=\vg \cdot \beta_i(v_i)=\vg \cdot f(\alpha_i(v_i)),\ \text{for}\ v_i \in V_i, \vg \in G.$$
Lastly, let $U\sset A$ be open. By continuity of $\beta_i$, $\beta_i^{-1}(U) \sset V_i$ is open, for every $i \in \I$. Since the $V_i$'s appear amongst the $(W, \tau_W)$, then
$\beta_i^{-1}(U) \x  \prod_{(W, \tau_W), \\ W\neq V_i}W$ is open in $\prod_{(W, \tau_W)} W$ and
$$q^{-1}(\beta_i^{-1}(U) \x \prod\limits_{\substack{(W, \tau_W), \\ W\neq V_i}}W)=q^{-1}(q^i(\beta_i^{-1}(U)))=\alpha_i(\beta_i^{-1}(U))=f^{-1}(U), $$
where $q^i$ is given componentwise by the $q^i_W$ defined above. Hence, $f$ is continuous and $\widebar{V}$ is indeed the coproduct of $\{ V_i\}_{i \in \I}$.
\end{proof}

With notation as before, we have the following diagram:

\begin{center}
\begin{tikzcd}
V_i \arrow[r, "\alpha_i"] 
\arrow[rr, bend right, "q^i"]  
& V \arrow[r, "q"] 
  & \prod\limits_{(W, \tau_W)}W.
\end{tikzcd}
\end{center}
Since $q^i_W$ is continuous for each $i \in \I$, then so is $q^i$ \cite{Bou}. By definition of the topology on $V$, $q$ is continuous. Hence, $\alpha_i$ is continuous for each $i$. This means that the topology on $V$ is contained in the final topology with respect to $\alpha_i$. However, the continuity of the $\alpha_i$ implies that $V$ appears as one of the $W$, thus, the coproduct topology defined above coincides with the final topology. 

Now we would like to give an explicit description of the basis of open neighbourhoods of $0$ in $V$. Chasco and Dom\'{i}nguez describe this basis with respect to the final topology for a coproduct of topological abelian groups \cite{CD}. 
We generalise their construction to topological $R$-modules: 

Let $\{U_i\}_{i \in \I}$ be a sequence of neighbourhoods of $0$, with $U_i$ a neighbourhood of $0$ in $V_i$. Let $\sJ \sset \I$ be finite. Then 
\begin{equation}\label{opens}
U \coloneq \bigcup_{\substack{n \in \N, \\ |\sJ|=n}} \bigcup_{\substack{\sK \sset \I, \\  |\sK|=|\sJ|}} \sum_{i \in \sK} \alpha_i(U_i)
\end{equation}
is a sequence of neighbourhoods of $0$ in $V$. Hence, the basis is given by
$$\sU = \{U\ |\ \{U_i\}_{i \in \I}, \ \text{with } U_i \sset V_i\ \text{open neighbourhood of } 0 \}.$$
This indeed agrees with our description of the topology:
If $\sB \in \prod_{(W, \tau_W)}W$ is an open neighbourhood of $0$ in $\prod_{(W, \t_W)}W$, $q^{-1}(\sB)$ would be the corresponding open in $V$ and



$$q^{-1}(\sB) \bigcap V= q^{-1}(\sB) \bigcap \bigoplus_{i \in \I} V_i= \bigoplus_{i \in \I} q^{-1}(\sB) \bigcap V_i,$$
can be written in the form of \ref{opens}.

\begin{thm} \label{adjoints-ltc}

Let $G$ be a topological group and $H \leq G$ a closed subgroup. Let 
$$\res^G_H: \M_{ltc}(G) \to \M_{ltc}(H)$$
be the restriction functor. The following hold:
\begin{enumerate}
\item $\res_H^G$ has a left adjoint $\coind_H^G: \M_{ltc}(H) \to \M_{ltc}(G)$.
\item If $H$ is also open, then $\res_H^G$ has a right adjoint 
$$\ind_H^G: \M_{ltc}(H) \to \M_{ltc}(G),$$
given by $\ind_H^G: W \mapsto \fun_H(G, W)$, where $\fun_H(G,W)$ is the space of all functions $f:G \to W$, such that $f(\vg \vh)=\vh \cdot f(\vg)$.
\end{enumerate}

\end{thm} 

\begin{proof}
We wish to apply Freyd's Adjoint Functor Theorem to prove the statements. First, $\M_{ltc}(G)$ is an abelian category. Hence, equializers and coequalizers exist. By the product-equalizer (coproduct-coequalizer) construction of limits (respectively colimits), Lemma \ref{prod} and Lemma \ref{coprod} imply that $\M_{ltc}(G)$ is both complete and cocomplete. Since the restriction functor does not change morphisms, it commutes with equalizers and coequalizers. Thus, to show existence of left and right adjoint to $\res^G_H$ we only need to check whether it commutes with products and coproducts.
Let us start with products, i.e., we want
$$\res^G_H(\prod_i V_i) \cong \prod_i \res^G_H(V_i).$$
By Lemma \ref{prod} $\prod_i \res^G_H(V_i) \in \M_{ltc}(H)$ and 
$$\prod_i \res^G_H(V_i) \cong \prod_i {}_{H}V_i \cong {}_{H}(\prod_i V_i)\cong \res^G_H(\prod_i V_i).$$
Note that ${}_{H}(\prod_i V_i)$ is already a continuous $H$-module. Thus, the second isomorphism holds because $\prod_i {}_{H}V_i$ and ${}_{H}(\prod_i V_i)$ have the same structure as $H$-modules, as well as topological spaces. This completes the proof of (1).

Now we move on to coproducts. 
To have a right adjoint to $\res^G_H$ we need to check cocontinuity, i.e.:
$$\bigoplus_i \res^G_H(V_i) \cong \res^G_H(\bigoplus_i V_i),$$
where by $\oplus$ we denote the coproduct in $\M_{ltc}(H)$ and $\M_{ltc}(G)$ respectively.
This amounts to showing that 
$$\bigoplus_i {}_{H}V_i \cong {}_{H} \big(\bigoplus_i V_i \big).$$
Since $\bigoplus_i {}_{H}V_i$ and ${}_{H}\bigoplus_i V_i$ have the same structure as $H$-modules and the action of $H$ on both is continuous, to obtain the isomorphism, we need to show that the coproduct topologies on both sides are the same.
In particular, for every $H$-module $(W, \t_W)$ for which there exists a surjective $H$-map $\varphi: \bigoplus_i {}_{H}V_i \to W$ with the property that the maps $q_i: V_i \to W$ are continuous, we have to find a module $(\widetilde{W}, \t_{\widetilde{W}}) \in \M_{ltc}(G)$, for which there exists a surjective $G$-map $\widetilde{\varphi}: {}_{H}\bigoplus_i V_i \to \widetilde{W}$, such that the maps $\widetilde{q_i}: V_i \to \widetilde{W}$ factoring through $\widetilde{\varphi}$ are continuous.
 In addition, for every open $U \sset W$ there must be an open $\widetilde{U} \sset \widetilde{W}$, such that $\widetilde{\varphi}^{-1}(\widetilde{U}) \sset \varphi(U)$.

Fix $(W, \t_{W})$ with the above properties. Let $\fun_H(G, W)$ be the space of all right $H$-equivariant functions $f: G \to W$, i.e., which satisfy the property:
$$ f(\vg \vh)=\vh \cdot f(\vg),\ \text{for}\ \vg \in G, \vh \in H.$$
This is a $G$-module via $\vg \cdot f: \vx \mapsto f(\vx \vg)$. We claim that $\widetilde{W} \coloneq \fun_H(G, W)$ satisfies the desired conditions. 
First, we show that $\widetilde{W} \in \M_{ltc}(G)$.
Let $X$ be a set of right coset representatives of $H$ in $G$. There is an $H$-module isomorphism:
$$\psi: \fun_H(G,W) \cong \prod_{a \in X} a \otimes_H W,\quad \psi: f \mapsto (a_i \otimes f(1))_{a_i \in X}.$$
Identifying the right hand-side as  $|X|$ copies of $W$, we can put the product topology on it. By Lemma \ref{prod}  $\prod_{a \in X} a \otimes_H W \in \M_{ltc}(H)$. 
Endow $\fun_H(G,W)$ with the topology induced by $\psi$. 
Since $\fun_H(G,W) \cong \prod_{a \in X} a \otimes_H W$ as $H$-modules and they have the same topology, then $\fun_H(G,W) \in \M_{ltc}(H)$. \newline
Let $U \sset \fun_H(G,W)$ be open. Suppose $\vg \cdot f \in U$, for $\vg \in G$ and $f \in \fun_H(G,W)$. 
The function $\vg \cdot f$ is $H$-equivariant and we can rewrite it as $\vg \cdot f: \vx \mapsto f(\vx \vg)=f(\vy \vh)=\vh \cdot f(\vy) \coloneq \vh \cdot w_{\vy} \in U$, for $\vg, \vx, \vy \in G$, $\vh \in H$ and $w_{\vy} \in W$. Since $W$ is a continuous representation of $H$, there exist an open neighbourhood $K$ of $1_H$ and an open submodule $Z \leq W$, such that $K \vh( w_{\vy}+Z)\in U$.
But as $H \leq G$ is open, $K$ is an open neighbourhood of $1_G$. Set $\widetilde{Z} \coloneq \psi^{-1}(1_G \otimes Z \x \prod_{a \in X, a \neq 1_G} a \otimes W)$. This is an open submodule of $\fun_H(G,W)$. Moreover, the pair $(K, \widetilde{Z})$ has the property $K \vg( f + \widetilde{Z}) \sset U$. In particular, $\fun_H(G,W)$ is a continuous representation of $G$.

Next, extend every surjective $H$-map $\varphi: \bigoplus_i {}_{H}V_i \to W$ to a surjective $G$-map $\widetilde{\varphi}: {}_{H}\bigoplus_i V_i \to \widetilde{W}$ by
$$\widetilde{\varphi}: v \mapsto f_v: \vg \to \varphi(\vg \cdot v).$$ 
Fix an open $U \sset W$. Then $b \otimes U \x \prod_{a \in X, a \neq b} a \otimes W \sset \prod_{a \in X} a \otimes W$ is open. Let $\widetilde{U}= \psi^{-1}(b \otimes U \x \prod_{a \in X, a \neq b} a \otimes W).$ By definition $\widetilde{U}$ is open in $\widetilde{W}$ and \newline
$\widetilde{U}=\{f \in \fun_H(G,W)\ |\ f(1) \in U \}.$ Then

$$\widetilde{\varphi}^{-1}(\widetilde{U}) = \{ v \in V\ |\ f_v \in \widetilde{U} \} 
= \{v \in V\ |\ f_v(1) \in U \} \sset \{v \in V\ |\ \varphi(v) \in U \} = \varphi^{-1}(U).$$

\end{proof}



\section{Category of Compact Representations} \label{three}

Let us start by defining the category $\M_c(G)$. As always the objects are pairs $(\pi, V)$, where $(V, \sT)$ is a topological $R$-module and $\pi: G \to \aut_R(V)$ is a continuous representation of a topological group $G$. We need to describe the topology $\sT$ on $V$. 

Firstly, we look at the case when $R= \F$ is a field, not necessarily commutative. 
As explained in the Intoduction, there are a few equivalent definitions of (linear) compactness for topological vector spaces. Following Beilinson-Drinfeld we view them as topological duals $V^{\star}$ of discrete vector spaces $V$ \cite{Dri}, \cite{Hitch}. By a topological dual we mean the space of all continuous linear functionals on $V$. The topology on $V^{\star}$ is given by orthogonal complements of finite dimensional subspaces of $V$ with respect to the canonical pairing \cite{Hitch}.

Define
$$\sD: \M_d(G) \to \M_c(G)$$
by
$$ V \mapsto V^{\star}\ \text{and}\ f \mapsto f^\star,\ \text{where}\ f^\star: \varphi \mapsto \varphi \circ f.$$
We claim that this is a contravariant functor, which induces an anti-equivalence of categories. 
First, let us check that $\sD$ is indeed a functor.

\begin{lemma} \label{dual}
Suppose $R=\F$ is a field and let $V \in \M_d(G)$. Then $V^\star \in \M_c(G)$, where $V^\star \coloneq \{ f: V \to \F\}$ is the space of all continuous linear functionals on $V$.
\end{lemma}

\begin{proof}
Let $\phi: G \x V^\star \to V^\star$, given by $(\vg, f) \mapsto \vg \cdot f: x \mapsto f(\vg^{-1}x)$ be the map defining the $G$-action on $V^\star$. This makes $V^\star$ into a $G$-module. We only need to show that $\phi$ is continuous.

Let $M^\star \sset V^\star$ be open. Suppose that $\vg \cdot f \in M^\star$, for some $\vg \in G$ and $f \in V^\star$. By definition $M^\star \coloneq \{ f:V \to \F\ |\ f(m)=0\ \text{for all}\ m \in M\}$, where $M \sset V$ is finite dimensional. Then $M=\F \<m_1,..,m_n \>$, for some $m_i \in V, i=1,...,n$. Since $V \in \M_d(G)$, there exist open neighbourhoods $K_i$ of $1_G$, such that $\vk_i \cdot m_i=m_i$, for every $\vk_i \in K_i$ and $i=1,..,n$. Since $G$ is endowed with a group topology, we can choose $K_i$ to be symmetric. Let $K=\bigcap_{i=1}^n K_i$. This is an open symmetric neighbourhood of $1_G$. Now consider $N \coloneq \vg \cdot M= \F \<\vg \cdot m_1,...,\vg \cdot m_n \>$. This is a finite dimensional subspace of $V$. Hence, $N^\star \coloneq \{ f:V \to \F\ |\ f(n)=0\ \text{for all}\ n \in N\} \sset V^\star$ is open. 
Thus, $K \sset G$ and $N^\star \sset V^\star$ are both open and $\vg K \cdot (f + N^\star) \sset M^\star$, finishing the proof.
\end{proof}

Lemma \ref{dual} shows that $\sD$ maps objects to objects. Let us check it does the same on morphisms. Let $V_1, V_2 \in \M_d(G)$ and $f: V_1 \to V_2$ be a morphism. Then $f^\star: V_2^\star \to V_1^\star$ has the following properties:
\begin{enumerate}
\item $f^\star( \vg \cdot \varphi)= (\vg \cdot \varphi) \circ f: v_1 \mapsto \varphi( \vg^{-1} f(v_1))=\varphi(f(\vg^{-1}v_1))=\varphi \circ (\vg \cdot f)= \vg \cdot f^\star(\varphi)$, where $\vg \in G, v_1 \in V_1, \varphi \in V_2^\star$. Thus $f^\star$ is $G$-linear.
\item Let $U^\star \sset V_1^\star$ be open. Then $f^\star(U^\star)^{-1}=\{\varphi \in V_2^\star\ |\ (\varphi \circ f)(U)=0\}$, for $U \sset V_1$ of finite dimension. But since $f$ is a linear map, then $f(U) \sset V_2$ is also a finite dimensional subspace. By definition of the topology on $V_2^\star$ it follows that $f^\star(U^\star)^{-1}$ is open and $f^\star$ is continuous. 
\end{enumerate}

Therefore, $\sD$ is indeed a functor. It is clear from the definition of a compact vector space that $\sD$ is bijective. Thus:

\begin{lemma} \label{prod->coprod}
The functor $\sD$ induces an anti-equivalence between $\M_d(G)$ and $\M_c(G)$. In particular, $\sD$ maps products in $\M_d(G)$ to coproducts in $\M_c(G)$.
\end{lemma}
In the following example we use our observations about $\sD$ to gain information about the cocontinuity of $\res_H^G$ in $\M_c(G)$.
\begin{ex} \label{ex-cpct}
 Suppose $\{ V_i \}_{i \in \I}$ is an arbitrary collection of discrete vector spaces over a field $\F$, such that $G$ acts continuously on each $V_i$. 
By Lemma \ref{dual} $V_i^\star \in \M_c(G)$. By Freyd's Theorem to have a right adjoint to $\res_H^G$ in $\M_c(G)$ we need 
\begin{equation} \label{isom-c}
\res_H^G \big(\bigoplus_{i\in \I} V_i^\star \big) \cong \bigoplus_{i\in \I} \res_H^G(V_i^\star),
\end{equation}

where $\bigoplus$ denotes the coproduct. By Lemma \ref{prod->coprod}
\begin{equation} \label{1}
\res_H^G\big(\bigoplus_{i\in \I} V_i^\star \big) \cong \res_H^G\big(\big(\big(\prod_{i \in \I} V_i \big)^{sm} \big)^\star \big).
\end{equation}
However, 

\begin{equation} \label{2}
\bigoplus_{i\in \I} \res_H^G(V_i^\star) \cong \big((\prod_{i \in \I} {}_{H}V_i)^{sm}\big)^\star,
\end{equation}
By Theorem \ref{l-adj d} 
$\res_H^G\big(\big(\prod_{i \in \I} V_i \big)^{sm} \big) \cong \prod_{i \in \I} \big(({}_{H}V_i)^{sm}\big)$ if $H$ is open. Thus, for $H\leq G$ open \ref{isom-c} holds and there is a well-defined functor 
$$\ind_H^G: \M_c(H) \to \M_c(G).$$
\end{ex}

As explained in the Introduction, all linearly compact vector spaces are complete. This means that $\M_c(G)$ is a subcategory of $\M_{ltc}(G)$. In the next example we consider the coproduct $\bigoplus_{i\in \I} V_i^\star $, for $V_i^\star \in \M_c(G)$, as an object of $\M_{ltc}(G)$. We aim to illustrate that in the case of topological vector spaces, if $H \leq G$ is not open, $\ind_H^G$ is not always defined in $\M_{ltc}(G)$, too.
\begin{ex}
For a collection $\{V_i^\star \}_{i \in \I} \in \M_c(G)$ defined as in Example \ref{ex-cpct}, consider their coproduct $\bigoplus_{i \in \I} V_i^\star$ in $\M_{ltc}(G)$. Keeping the notation and conventions of Section \ref{two}, recall that this is the coproduct in \textbf{Vect}$_\F$ with topology induced by the embedding
$$ q: \bigoplus_{i \in \I} V_i^\star \to \prod_{(W, \t_W)} W.$$
Since we are considering the coproduct as an object of $\M_{ltc}(G)$, rather than $\M_c(G)$, some of the $(W, \t_W)$ can be linearly topologized and complete, but not compact. Suppose that this is the case. The continuity of the maps $q_W^i:V_i^\star \to W$ implies that for every open submodule $U \sset W$ and for all $i \in \I,  q^i_W(U)^{-1}$ is open and hence of finite codimension in $V_i^\star$.
Since $q_W^i=q_W \circ \alpha_i$ and $q=(q_W)_{(W, \t_W)}$, it follows that $q^{-1}(U)$ is of finite codimension in $\bigoplus_{i \in \I} V_i^{\star}$.
As all open submodules of $\bigoplus_{i \in \I} V_i^{\star}$ correspond to inverse images of open submodules $U \leq \prod_{(W, \t_W)}W$, the topology on $\bigoplus_{i \in \I} V_i^\star$ is compact.
By the uniqueness of coproducts and Example \ref{ex-cpct} it follows that if $H$ is not open, the restriction functor $\res_H^G: \M_{ltc}(G) \to \M_{ltc}(H)$ is not cocontinuous.
\end{ex}


We move on to $R$-modules, where $R$ is an associative ring with 1. We call an $R$-module $V$ \emph{compact} if $V$ admits a linear complete topology with the additional property that if $U \sset V$ is an open submodule, then $V/U$ is of finite length. Such modules are sometimes called \emph{pseudocompact} \cite{Iovanov}, \cite{frauke}. We call a topological $R$-module $V$ \emph{linearly compact} if it is linearly topologized, Hausdorff, and such that every family of closed cosets in $V$ has the finite intersection property \cite{War}. Every compact module is linearly compact \cite{Iovanov}. We denote by $\M_c(G)$ the category of all compact $R$-modules which admit a continuous action of the topological group $G$. We now construct products and coproducts in $\M_c(G)$.

\begin{lemma} \label{prod-c}
Arbitrary products exist in $\M_c(G)$.
\end{lemma}
\begin{proof}
A product of linearly compact $R$-modules is linearly compact with respect to the product topology \cite{War}. Since every compact module is linearly compact, then the category of compact modules is closed under products.

By exactly the same argument as in Lemma \ref{prod} for an arbitrary collection $\{V_i \}_{i \in \I}$ of elements of $\M_c(G)$ the product $V \coloneq \prod_{i \in \I} V_i$ in $R$-\textbf{mod}, endowed with the product topology, is a continuous $G$-module with respect to the componentwise action of $G$. 

\end{proof}

We now wish to form coproducts in $\M_{c}(G)$. 
For an arbitrary collection $\{V_i \}_{i \in \I} \in \M_{c}(G)$, we form the coproduct $V\coloneq \bigoplus_{i \in \I} V_i$ in $R$-\textbf{mod}. To define a topology $\sT_V$ on $V$ we mimic the procedure from Section \ref{two}:
Let $W \in \M_c(G)$. Suppose there exists a surjective $R$-linear map $q_W: V \to W$ which commutes with the $G$-action, such that the maps $q^i_W : V_i \to W$, factoring through $q_W$, are continuous. The topology $\sT_V$ on $V$ is induced by the embedding

\begin{equation} \label{emb-c}
q: V \to \prod_{(W, \t_W)}W, \quad v \mapsto (q_W(v))_{(W, \t_W)}.
\end{equation}

\begin{lemma} \label{coprod-c}
Let $\{V_i \}_{i \in \I}$ be an arbitrary collection of elements of $\M_c(G)$. Their coproduct is the module $(V, \sT_V)$ described above. 
\end{lemma}

\begin{proof}
By Lemma \ref{prod-c} $\prod_{(W, \t_w)}W \in \M_{ltc}(G)$. Since $V \sset \prod_{(W, \t_w)}W$, the topology on $V$ is linear \cite{Lef}. Let $U \in \prod_{(W, \t_w)}W$ be a basic open. Then $U=U_i \x \prod_{(W, \t_w)}W$, where $U_i \sset W$ for some $W$, is an open submodule. By definition $q^{-1}(U)$ is open in $V$. But $q^{-1}(U)=\bigoplus_{i \in \I} q^{-1}(U) \cap V_i$. Since each $q^i_W: V_i \to W$ is continuous, it follows that $q^{-1}(U) \cap V_i$ is an open submodule of $V_i$. Hence, the quotient is of finite length. This implies that $V /q^{-1}(U)$ is also of finite length, showing that the topology $\sT_V$ is compact.
Every compact space is linearly compact. Thus, $V$ is closed in $\prod_{(W, \t_W)}W$ \cite{War}. In particular, $V$ is complete \cite{Bou}.
The map $G \x \prod_{(W, \t_W)} W \to \prod_{(W, \t_W)} W$ is continuous, and thus so is its restriction to a subspace, i.e., $V \in \M_c(G)$. As $\M_c(G) \sset \M_{ltc}(G)$ and the coproducts in the two categories are constructed in the same way, Lemma \ref{coprod} implies that $V$ satisfies the universal property of the coproduct in $\M_c(G)$, finishing the proof.

\end{proof}
Having constructed products and coproducts in $\M_c(G)$, in order to establish existence of left and right adjoint to the restriction functor $\res_H^G$ in $\M_c(G)$, we need to check whether it is continuous and cocontinuous.

\begin{thm} \label{adj-c}
Let $G$ be a topological group and $H \leq G$ a closed subgroup. The restriction functor
$$\res_H^G: \M_c(G) \to \M_c(H)$$
has a left adjoint. Hence, we have a well-defined $\coind_H^G: \M_c(H) \to \M_c(G)$. It has a right adjoint if $H$ is open.
\end{thm}
\begin{proof}
Since $\M_c(G)$ is an abelian subcategory of $\M_{ltc}(G)$ and the products and coproducts in $\M_c(G)$ are the same as in $\M_{ltc}(G)$, then the statement is just a corollary to Theorem \ref{adjoints-ltc}. 
\end{proof}


Having already studied the categories $\M_c(G)$ and $\M_{ltc}(G)$, there is another category of topological vector spaces which is interesting to consider. We call a vector space \emph{Tate}, if it splits as a direct sum of a discrete and a compact space. Thus, we can form a category $\M_T(G)$ of Tate spaces on which $G$ acts continuously. 
However, if one would like to investigate Frobenius reciprocity in $\M_T(G)$, one encounters a difficulty straight away - the category is not even closed under products. Therefore, we propose to look at the free product completion $\widebar{\M_T(G)}$ of $\M_T(G)$. Since products and coproducts in $\M_{ltc}(G)$ exist and $\M_T(G) \sset \M_{ltc}(G)$, then $\widebar{\M_T(G)}$ would also be a subcategory of $\M_{ltc}(G)$. If the products in the two categories coincide, then by Theorem \ref{adjoints-ltc} the restriction functor in $\widebar{\M_T(G)}$ would always have a left adjoint. An interesting question in this case would be whether for $(\s, W) \in \widebar{\M_T(H)}$, there exists a canonical $G$-submodule of $\coind_H^G(\s)$ which is Tate.
Coproducts and induction can be approached similarly. Firstly, we construct a free coproduct completion $\widehat{\M_T(G)}$ of $\M_T(G)$. If coproducts there coincide with coproducts in $\M_{ltc}(G)$, then we expect to have existence of an induction functor in $\widehat{\M_T(G)}$ in the case when $H \leq G$ is also open.

\end{document}